\documentclass{article}
\usepackage[english]{babel} 		
\usepackage{fullpage}				
\usepackage[utf8]{inputenc}			
\usepackage{amsmath, amssymb,amsthm}
\usepackage{graphicx}				
\usepackage{hyperref}				
\usepackage{xcolor}
\usepackage{multirow}
\usepackage{tikz}
\usepackage{ytableau}
\usepackage{mathrsfs}
\usepackage{thmtools}
\usepackage{thm-restate}
\usepackage{varwidth}
\usepackage{comment}
\usepackage{mathtools}
\usepackage[mathscr]{euscript}
\usepackage{bbm}
\usepackage{multicol}
\usepackage{enumerate}

\usepackage[
    backend=bibtex,
    maxnames=5,
    style=alphabetic
  ]{biblatex}
\theoremstyle{definition}
\newtheorem{thm}{Theorem}[section]
\newtheorem{lem}[thm]{Lemma}
\newtheorem*{lem*}{Lemma}
\newtheorem*{thm*}{Theorem}
\newtheorem{prop}[thm]{Proposition}

\newtheorem{cor}[thm]{Corollary}

\newtheorem{defn}[thm]{Definition}
\newtheorem*{remark*}{Remark}
\newtheorem{remark}{Remark}
\newtheorem{example}{Example}

\newtheorem{cor/defn}[thm]{Corollary/Definition}
\DeclareMathOperator{\Par}{\mathbb{Y}}
\DeclareMathOperator{\RYT}{\mathrm{RYT}_{\geq 0}}
\DeclareMathOperator{\SYT}{\mathrm{SYT}}
\DeclareMathOperator{\PSYT}{\mathrm{PSYT}_{\geq 0}}
\DeclareMathOperator{\RSSYT}{\mathrm{RSSYT}_{\geq 0}}
\DeclareMathOperator{\sH}{\mathscr{H} }
\DeclareMathOperator{\sA}{\mathscr{A} }
\DeclareMathOperator{\sD}{\mathscr{D} }

\DeclareMathOperator{\Ind}{Ind}
\DeclareMathOperator{\Res}{Res}

\DeclareMathOperator{\MacD}{\mathfrak{P}}
\DeclareMathOperator{\sE}{\mathscr{E}}
\DeclareMathOperator{\Inv}{\mathrm{Inv}}
\DeclareMathOperator{\inv}{inv}
\DeclareMathOperator{\rk}{\mathrm{rk}}
\DeclareMathOperator{\APSYT}{\mathrm{APSYT}_{\geq 0}}
\DeclareMathOperator{\ASYT}{\mathrm{ASYT}}
\DeclareMathOperator{\Top}{\mathrm{top}}
\DeclareMathOperator{\Min}{\mathrm{min}}
\DeclareMathOperator{\I}{\mathrm{I}}
\DeclareMathOperator{\Hilb}{\mathrm{Hilb}}
\DeclareMathOperator{\SL}{\mathrm{SL}}
\DeclareMathOperator{\End}{\mathrm{End}}

\title{Murnaghan-Type Representations of the Elliptic Hall Algebra}
\author{Milo James Bechtloff Weising}
\date{\today}

\begin{document}

\maketitle

\abstract{We construct a new family of graded representations $\widetilde{W}_{\lambda}$ indexed by Young diagrams $\lambda$ for the positive elliptic Hall algebra $\mathcal{E}^{+}$ which generalizes the standard $\mathcal{E}^{+}$ action on symmetric functions. These representations have homogeneous bases of eigenvectors for the action of the Macdonald element $P_{0,1} \in \mathcal{E}^{+}$ generalizing the symmetric Macdonald functions. The analysis of the structure of these representations exhibits interesting combinatorics arising from the stable limits of periodic standard Young tableaux. We find an explicit combinatorial rule for the action of the multiplication operators $e_r[X]^{\bullet}$ generalizing the Pieri rule for symmetric Macdonald functions. Lastly, we obtain a family of interesting $q,t$ product-series identities which come from the analysis of certain combinatorial statistics associated to periodic standard Young tableaux.}

\tableofcontents

\section{Introduction}
This is a version of the author’s FPSAC 2024 submission. For the sake of satisfying the page limit for FPSAC most of the proofs are omitted. The complete version of this paper with full proofs will appear in the coming months.

The space of symmetric functions, $\Lambda$, is a central object in algebraic combinatorics deeply connecting the fields of representation theory, geometry, and combinatorics. In his influential paper \cite{MacDSLC}, Macdonald introduced a special basis $P_{\lambda}[X;q,t]$ for $\Lambda$ over $\mathbb{Q}(q,t)$ simultaneously generalizing many other important and well-studied symmetric function bases like the Schur functions $s_{\lambda}[X]$. These symmetric functions $P_{\lambda}[X;q,t]$, called the symmetric Macdonald functions, exhibit many striking combinatorial properties and can be defined as the eigenvectors of a certain operator $\Delta: \Lambda \rightarrow \Lambda$ called the Macdonald operator constructed using polynomial difference operators. It was discovered through the works of Bergeron, Garsia, Haiman, Tesler, and many others \cite{haiman2000hilbert} \cite{BGSciFi} \cite{BGHT} that variants of the symmetric Macdonald functions called the modified Macdonald functions $\widetilde{H}_{\lambda}[X;q,t]$ have deep ties to the geometry of the Hilbert schemes $\Hilb_n(\mathbb{C}^2).$ On the side of representation theory, it was shown first in full generality by Cherednik \cite{C_2001} that one can recover the symmetric Macdonald functions by considering the representation theory of certain algebras called the spherical double affine Hecke algebras (DAHAs) in type $GL_n.$

The positive elliptic Hall algebra (EHA), $\sE^{+}$, was introduced by Burban and Schiffmann \cite{BS} as the positive subalgebra of the Hall algebra of the category of coherent sheaves on an elliptic curve over a finite field. This algebra has connections to many areas of mathematics including, most importantly for the present paper, to Macdonald theory. In \cite{SV}, Schiffmann and Vasserot realize $\mathcal{E}^{+}$ as a stable limit of the positive spherical DAHAs in type $GL_n$. They show further that there is a natural action of $\mathcal{E}^{+}$ on $\Lambda$ aligning with the spherical DAHA representations originally considered by Cherednik. In particular, the action of $P_{0,1} \in \sE^{+}$ gives the Macdonald operator $\Delta$. The action of $\sE^{+}$ on $\Lambda$ can be realized as the action of certain generalized convolution operators on the torus equivariant $K$-theory of the schemes $\Hilb_n(\mathbb{C}^2).$

Dunkl and Luque in \cite{DL_2011} introduced symmetric and non-symmetric vector-valued (vv.) Macdonald polynomials. The term vector-valued here refers to polynomial-like objects of the form $\sum_{\alpha}c_{\alpha}X^{\alpha}\otimes v_{\alpha}$ for some scalars $c_{\alpha}$, monomials $X^{\alpha}$, and vectors $v_{\alpha}$ lying in some $\mathbb{Q}(q,t)$-vector space. The non-symmetric vv. Macdonald polynomials are distinguished bases for certain DAHA representations built from the irreducible representations of the finite Hecke algebras in type A. These DAHA representations are indexed by Young diagrams and exhibit interesting combinatorial properties relating to periodic Young tableaux. The symmetric vv. Macdonald polynomials are distinguished bases for the spherical (i.e. Hecke-invariant) subspaces of these DAHA representations. Naturally, the spherical DAHA acts on this spherical subspace with the special element $Y_1+\ldots + Y_n$ of spherical DAHA acting diagonally on the symmetric vv. Macdonald polynomials. 

Dunkl and Luque in \cite{DL_2011} (and in later work of Colmenarejo, Dunkl, and Luque \cite{CDL_2022} and Dunkl \cite{D_2019}) only consider the finite rank non-symmetric and symmetric vv. Macdonald polynomials. It is natural to ask if there is an infinite-rank stable-limit construction using the symmetric vv. Macdonald polynomials to give generalized symmetric Macdonald functions and an associated representation of the positive elliptic Hall algebra $\mathcal{E}^{+}$. In this paper, we will describe such a construction (Thm. \ref{main theorem}). We will obtain a new family of graded $\mathcal{E}^{+}$-representations $\widetilde{W}_{\lambda}$ indexed by Young diagrams $\lambda$ and a natural generalization of the symmetric Macdonald functions $\MacD_{T}$ indexed by certain labellings of infinite Young diagrams built as limits of the symmetric vv. Macdonald polynomials. For combinatorial reasons there is essentially a unique natural way to obtain this construction. For any $\lambda$ we will consider the increasing chains of Young diagrams $\lambda^{(n)} = (n-|\lambda|,\lambda)$ for $n \geq |\lambda| + \lambda_1$ to build the representations $\widetilde{W}_{\lambda}$. These special sequences of Young diagrams are central to Murnaghan's theorem \cite{M_1938} regarding the reduced Kronecker coefficients. As such we refer to the $\mathcal{E}^{+}$-representations $\widetilde{W}_{\lambda}$ as Murnaghan-type. For $\lambda = \emptyset$ we recover the $\mathcal{E}^{+}$ action on $\Lambda$ and the symmetric Macdonald functions $P_{\mu}[X;q,t]$. We will show that these Murnaghan-type representations $\widetilde{W}_{\lambda}$ are mutually non-isomorphic. The existence of these representations of the elliptic Hall algebra raises many questions about possible new relations between Macdonald theory and geometry. Other authors have constructed families of $\sE^{+}$-representations \cite{FFJMM_2011} \cite{FJMM_2012}. Although there should exist a relationship between the Murnaghan-type representations $\widetilde{W}_{\lambda}$ and those of other authors, the construction in this paper appears to be distinct from prior $\sE^{+}$-module constructions. 

For technical reasons regarding the misalignment of the spectrum of the Cherednik operators $Y_i$ we will need to restate many of the results of Dunkl and Luque in \cite{DL_2011} using a re-oriented version of the Cherednik operators $\theta_i$. This alternative choice of conventions greatly assists during the construction of the generalized Macdonald functions $\MacD_{T}$. The combinatorics underpinning the non-symmetric vv. Macdonald polynomials originally defined by Dunkl and Luque will be reversed in the conventions appearing in this paper. 

\begin{subsection}{Overview}
    Here we will give a brief overview of this paper. First, in Section \ref{defs and nots} we will review relevant definitions and notations as well as recall the stable-limit spherical DAHA construction of Schiffmann-Vasserot. In Section \ref{DAHA Modules from Young Diagrams} we will re-state many of the results of Dunkl-Luque but for the re-oriented Cherednik operators including describing the non-symmetric vv. Macdonald polynomials $F_{\tau}$ and their associated Knop-Sahi relations (Prop. \ref{weight basis prop}). We define (Def. \ref{connecting maps def}) the DAHA modules $V_{\lambda^{(n)}}$ and connecting maps $\Phi^{(n)}_{\lambda}: V_{\lambda^{(n+1)}} \rightarrow V_{\lambda^{(n)}}$ which will be used in the stable-limit process. Next in Section \ref{Positive EHA Representations from Young Diagrams}, we describe the spherical subspaces $W_{\lambda^{(n)}}$ of Hecke invariants of $V_{\lambda^{(n)}}$ and the symmetric vv. Macdonald polynomials $P_{T}$ including an explicit expansion of the $P_{T}$ into the $F_{\tau}$ (Prop. \ref{expansion of sym into nonsym}). We will use the connecting maps to define the stable-limit spaces $\widetilde{W}_{\lambda}$ and show in Thm. \ref{main theorem} that they possess a graded action of $\mathcal{E}^{+}$ and have a distinguished basis of generalized symmetric Macdonald functions $\MacD_{T}.$ In Section \ref{Pieri Rule section} we will obtain a Pieri formula (Cor. \ref{Pieri Rule}) for the action of $e_r[X]^{\bullet}$ on the generalized Macdonald functions $\MacD_T$. Lastly in Section \ref{Product-Sum Identities Sections}, we will look at an interesting family of $(q,t)$ product-series identities (Thm. \ref{prod-sum formula}) which follow naturally from the combinatorics in the prior sections of the paper.
\end{subsection}

\begin{subsection}{Acknowledgements}
The author would like to thank their advisor Monica Vazirani for her consistent guidance. The author would also like to thank Erik Carlsson, Daniel Orr, and Eugene Gorsky for helpful conversations about the elliptic Hall algebra and the geometry of Hilbert schemes. The author was supported during this work by the 2023 UC Davis Dean's Summer Research Fellowship.
\end{subsection}

\section{Definitions and Notations}\label{defs and nots}

\subsection{Some Combinatorics}
We start with a description of many of the combinatorial objects which we will need for the remainder of this paper.

\begin{defn}\label{Tableaux defs}
    A \textit{\textbf{partition}} is a (possibly empty) sequence of weakly decreasing positive integers. Denote by $\Par$ the set of all partitions. Given a partition $\lambda = (\lambda_1,\ldots, \lambda_r)$ we set $\ell(\lambda) := r$ and $|\lambda| := \lambda_1 + \ldots + \lambda_r.$ For $\lambda = (\lambda_1,\ldots, \lambda_r) \in \Par$ and $n \geq n_{\lambda}:= |\lambda| + \lambda_1$ we set $\lambda^{(n)}:= (n -|\lambda|, \lambda_1,\ldots, \lambda_r).$ We will identify partitions as defined above with \textit{\textbf{Young diagrams}} of the corresponding shape in English notation i.e. justified up and to the left. 
    
    Fix a partition $\lambda$ with $|\lambda| = n$. We will require each of the following combinatorial constructions for types of labelling of the Young diagram $\lambda$. If a diagram $\lambda$ appears as the domain of a labelling function then we are referring to the set of boxes of $\lambda$ as the domain.
\begin{itemize}
    \item A non-negative \textit{\textbf{reverse Young tableau}} $\RYT(\lambda)$ is a labelling $T: \lambda \rightarrow \mathbb{Z}_{\geq 0}$ which is weakly decreasing along rows and columns.
    \item A non-negative \textit{\textbf{reverse semi-standard Young tableau}} $\RSSYT(\lambda)$ is a labelling $T: \lambda \rightarrow \mathbb{Z}_{\geq 0}$ which is weakly decreasing along rows and strictly decreasing along columns.
    \item A \textit{\textbf{standard Young tableau}} $\SYT(\lambda)$ is a labelling $\tau:\lambda \rightarrow \{1,\ldots,n\}$ which is strictly increasing along rows and columns.
    \item A non-negative \textit{\textbf{periodic standard Young tableau}} $\PSYT(\lambda)$ is a labelling $\tau: \lambda \rightarrow \{ jq^{b}: 1\leq j \leq n, b \geq 0\}$ in which each $1\leq j \leq n$ occurs in exactly one box of $\lambda$ and where the labelling is strictly increasing along rows and columns. Here we order the formal products $jq^m$ by $jq^m < kq^{\ell}$ if $m > \ell$ or in the case that $m = \ell$ we have $j < k.$ Note that $SYT(\lambda) \subset \PSYT(\lambda)$.
\end{itemize}

\end{defn}

\begin{example}
  
\ytableausetup{centertableaux, boxframe= normal, boxsize= 2.25em}
\begin{ytableau}
 17q^7 & 15q^5 & 16q^5 & 11q^3 & 7q^1 & 2q^0 \\
 14q^6 & 12q^4 & 13q^4 & 9q^2 & 8q^0 & \none \\
 10q^2 & 4q^1 & 5q^1 & 6q^1 & \none & \none \\
 3q^1 & 1q^0 & \none & \none & \none & \none \\
\end{ytableau} $\in \PSYT(6,5,4,2)$

\end{example}

\begin{defn}\label{ordering and operations on tableaux defs}
Given a box, $\square$, in a Young diagram $\lambda$ we define the content of $\square$ as $c(\square) := a-b$ where $\square = (a,b)$ as drawn in the $\mathbb{N}\times \mathbb{N}$ grid. Let $\tau \in \PSYT(\lambda)$ and $1\leq i\leq n$. Whenever $\tau(\square) = iq^b$ for some box $\square \in \lambda$ we will write $c_{\tau}(i):= c(\square)$ and $w_{\tau}(i):= b.$ Let $1\leq j \leq n-1$ and suppose that for some boxes $\square_1,\square_2 \in \lambda$ that $\tau(\square_1) = jq^m$ and $\tau(\square_2) = (j+1)q^{\ell}$. Let $\tau'$ be the labelling defined by $\tau'(\square_1) = (j+1)q^m$, $\tau'(\square_2) = jq^{\ell}$, and $\tau'(\square) = \tau(\square)$ for $\square \in \lambda \setminus \{\square_1,\square_2\}$. If $\tau' \in \PSYT(\lambda)$ then we write $s_j(\tau):= \tau'$. Let $\Psi(\tau) \in \PSYT(\lambda)$ be the labelling defined by whenever $\tau(\square) = kq^a$ then either $\Psi(\tau)(\square) = (k-1)q^a$ when $k \geq 2$ or $\Psi(\tau)(\square) = nq^{a+1}$ when $k = 1.$ We give the set $\PSYT(\lambda)$ a partial order defined by the following cover relations.
\begin{itemize}
    \item For all $\tau \in \PSYT(\lambda)$, $\Psi(\tau) > \tau.$
    \item If $w_{\tau}(i)<w_{\tau}(i+1)$ then $s_i(\tau) > \tau.$
    \item If $w_{\tau}(i) = w_{\tau}(i+1)$ and $c_{\tau}(i)-c_{\tau}(i+1) > 1$ then $s_i(\tau) > \tau.$
\end{itemize}
Define the map $\mathfrak{p}_{\lambda}: \PSYT(\lambda) \rightarrow \RYT(\lambda)$ by $\mathfrak{p}_{\lambda}(\tau)(\square) = b$ whenever $\tau(\square) = iq^b.$ We will write $\PSYT(\lambda;T)$ for the set of all $\tau \in \PSYT(\lambda)$ with $\mathfrak{p}_{\lambda}(\tau) = T \in \RYT(\lambda).$
\end{defn}

\begin{example} 
$\Psi \left( \ytableausetup{centertableaux, boxframe= normal, boxsize= 2.25em}
\begin{ytableau}
 1q^7& 3q^5 & 5q^5 & 8q^2 & 12q^1 & 17q^0 \\
 2q^6& 4q^5 & 6q^5 & 14q^0 & 16q^0 & \none \\
 7q^2& 10q^1 & 11q^1 & 15q^0 & \none & \none \\
 9q^1& 13q^0 & \none & \none & \none & \none \\
\end{ytableau} \right) = \ytableausetup{centertableaux, boxframe= normal, boxsize= 2.25em}
\begin{ytableau}
 17q^8& 2q^5 & 4q^5 & 7q^2 & 11q^1 & 16q^0 \\
 1q^6& 3q^5 & 5q^5 & 13q^0 & 15q^0 & \none \\
 6q^2& 9q^1 & 10q^1 & 14q^0 & \none & \none \\
 8q^1& 12q^0 & \none & \none & \none & \none \\
\end{ytableau}$

\end{example}

\begin{lem}\label{locally maximal periodic tableau}
    Let $\lambda \in \Par$ and $T \in \RYT(\lambda).$ There are unique $\Min(T), \Top(T) \in \PSYT(\lambda;T)$ such that for all $\tau \in \PSYT(\lambda)$ with $\mathfrak{p}_{\lambda}(\tau) = T$, $\Min(T) \leq \tau \leq \Top(T).$ 
\end{lem}

\begin{example}\label{example of labelling types}
    Given
  $T = $  \ytableausetup{centertableaux, boxframe= normal, boxsize= 2.25em}
\begin{ytableau}
 7& 5 & 5 & 2 & 1 & 0 \\
 6& 5 & 5 & 0 & 0 & \none \\
 2& 1 & 1 & 0 & \none & \none \\
 1& 0 & \none & \none & \none & \none \\
\end{ytableau} $\in \RYT(6,5,4,2)$
we have that 

$\Min(T) = $\ytableausetup{centertableaux, boxframe= normal, boxsize= 2.25em}
\begin{ytableau}
 17q^7& 12q^5 & 13q^5 & 10q^2 & 6q^1 & 1q^0 \\
 16q^6& 14q^5 & 15q^5 & 2q^0 & 3q^0 & \none \\
 11q^2& 7q^1 & 8q^1 & 4q^0 & \none & \none \\
 9q^1& 5q^0 & \none & \none & \none & \none \\
\end{ytableau}
and 
$\Top(T) = $\ytableausetup{centertableaux, boxframe= normal, boxsize= 2.25em}
\begin{ytableau}
 1q^7& 3q^5 & 5q^5 & 8q^2 & 12q^1 & 17q^0 \\
 2q^6& 4q^5 & 6q^5 & 14q^0 & 16q^0 & \none \\
 7q^2& 10q^1 & 11q^1 & 15q^0 & \none & \none \\
 9q^1& 13q^0 & \none & \none & \none & \none \\
\end{ytableau}.

\end{example}

\begin{defn}\label{decomposing RYT into SYT and partition}
Let $\lambda \in \Par$ with $|\lambda| = n$ and $T \in \RYT(\lambda).$ Define $\nu(T) \in \mathbb{Z}_{\geq 0}^{n}$ to be the vector formed by listing the values of T in decreasing order. Define $S(T) \in \SYT(\lambda)$ by ordering the boxes of $\lambda$ according to $\square_1 \leq \square_2$ if and only if 
\begin{itemize}
    \item $T(\square_1) > T(\square_2)$ or
    \item $T(\square_1) = T(\square_2)$ and $\square_1$ comes before $\square_2$ in the column-standard labelling of $\lambda.$
\end{itemize}
Define the statistic $b_T \in \mathbb{Z}_{\geq 0}$ by 
$$ b_T:= \sum_{i=1}^{n} \nu(T)_i( c_{S(T)}(i) + i-1).$$ 
Lastly, define the composition $\mu(T)$ of $n$ so that the Young subgroup $\mathfrak{S}_{\mu(T)}$ of $\mathfrak{S}_n$ is the stabilizer subgroup of $\Min(T)$ i.e. the group generated by the $s_i \in \mathfrak{S}_n$ such that the entries $iq^a$ and $(i+1)q^b$ occur in the same row of $\Min(T)$ for some $a,b \geq 0.$
\end{defn}

\begin{remark}
    For every $T \in \RYT(\lambda)$ we can recover $T$ from the pair $(S(T),\nu(T))$ by labelling $\lambda$ with the entries of $\nu(T)$ following the order of $S(T).$ Further, the standard Young tableau $S(T)$ is the largest such tableau following the partial order defined in Definition \ref{ordering and operations on tableaux defs}.
\end{remark}

\begin{example}
    For $T \in \RYT(6,5,4,2)$ as in Example \ref{example of labelling types} we have that 

 $S(T) = $\ytableausetup{centertableaux, boxframe= normal, boxsize= 2.25em}
\begin{ytableau}
 1& 3 & 5 & 8 & 12 & 17 \\
 2& 4 & 6 & 14 & 16 & \none \\
 7& 10 & 11 & 15 & \none & \none \\
 9& 13 & \none & \none & \none & \none \\
\end{ytableau} $\in \SYT(6,5,4,2),$

$\nu(T) = (7,6,5,5,5,5,2,2,1,1,1,1,0,0,0,0,0)\in \mathbb{Z}_{\geq 0}^{17},$
 
$b_T = 0+0+15+15+30+30+8+20+5+8+10+15+0+0+0+0+0 = 156,$

and $\mu(T) = (1,2,1,1,1,2,1,1,1,2,2,1,1).$
\end{example}

\begin{defn}\label{inversions defn}
    Let $\lambda \in \Par$, with $|\lambda|= n$ and $\tau \in \PSYT(\lambda)$ with $T = \mathfrak{p}_{\lambda}(\tau).$ An ordered pair of boxes $(\square_1,\square_2) \in \lambda \times \lambda$ is called an \textit{\textbf{inversion pair}} of $\tau$ if $S(T)(\square_1) < S(T)(\square_2)$ and $i > j$ where $\tau(\square_1) = iq^a$, $\tau(\square_2) = jq^b$ for some $a,b \geq 0.$
    The set of all inversion pairs of $\tau$ will be denoted by $\Inv(\tau).$ We will use the shorthand $\I(T)$ for the set $\Inv(\Min(T)).$
\end{defn}

\begin{example}
    In the labelling 
    \ytableausetup{centertableaux, boxframe= normal, boxsize= 2.25em}
\begin{ytableau}
 17q^7& 12q^5 & 13q^5 & 10q^2 & 6q^1 & 1q^0 \\
 16q^6& 14q^5 & 15q^5 & 2q^0 & 3q^0 & \none \\
 11q^2& 7q^1 & 8q^1 & 4q^0 & \none & \none \\
 9q^1& 5q^0 & \none & \none & \none & \none \\
\end{ytableau} 
we have that the pairs $(17q^7, 12q^5)$, $(14q^5,13q^5)$, and $(5q^0,4q^0)$ are all inversions. Here we have referred to boxes according to their labels.
\end{example}

\subsection{Finite Hecke Algebra}
Here we give a review of the finite Hecke algebras in type A and give a description of their irreducible representations.

\begin{defn}\label{finite hecke alg defn}
    Define the \textit{\textbf{finite Hecke algebra}} $\sH_n$ to be the $\mathbb{Q}(q,t)$-algebra generated by $T_1,\ldots, T_{n-1}$ subject to the relations 
    \begin{itemize}
        \item $(T_i-1)(T_i+t) = 0$ for $1\leq i \leq n-1$
        \item $T_iT_{i+1}T_1 = T_{i+1}T_iT_{i+1}$ for $1\leq i \leq n-2$
        \item $T_iT_j = T_jT_i$ for $|i-j| > 1.$
    \end{itemize}
    We define the special elements $\overline{\theta}_1,\ldots, \overline{\theta}_n \in \sH_n$ by $\overline{\theta}_1 := 1$ and $\overline{\theta}_{i+1}:= tT_{i}^{-1}\overline{\theta}_{i}T_{i}^{-1}$ for $1\leq i \leq n-1$. Further, define $\overline{\varphi}_1,\ldots, \overline{\varphi}_{n-1}$ by $\overline{\varphi}_i: = (tT_{i}^{-1})\overline{\theta}_i - \overline{\theta}_{i}(tT_i^{-1}).$ For a permutation $\sigma \in \mathfrak{S}_n$ and a reduced expression $\sigma = s_{i_1}\cdots s_{i_r}$ we write $T_{\sigma} := T_{i_1}\cdots T_{i_r}.$
\end{defn}

\begin{defn}\label{irreps for finite hecke defn}
    Let $\lambda \in \Par$ with $|\lambda| = n$. By an abuse of notation we will label by $\lambda$ the $\sH_n$-module spanned by $\tau \in \SYT(\lambda)$ defined by the following relations:
    \begin{itemize}
        \item $\overline{\theta}_i(\tau) = t^{c_{\tau}(i)}\tau$
        \item If $s_i(\tau) > \tau$ then $\overline{\varphi}_i(\tau) = (t^{c_{\tau}(i)}-t^{c_{\tau}(i+1)})s_i(\tau).$
        \item If the labels $i,i+1$ are in the same row in $\tau$ then $T_i(\tau) = \tau.$
        \item If the labels $i,i+1$ are in the same column in $\tau$ then $T_i(\tau) = -t\tau.$
    \end{itemize}
\end{defn}

\begin{lem}
    Let $\lambda \in \Par$, $n \geq n_{\lambda}$, and $\square_0 \in \lambda^{(n+1)}/\lambda^{(n)}$. There is a $\sH_{n}$-module map $\mathfrak{q}_{\lambda}^{(n)}:\lambda^{(n+1)}\rightarrow \lambda^{(n)}$ given by $$\mathfrak{q}_{\lambda}^{(n)}(\tau) := \begin{cases}
    \tau|_{\lambda^{(n)}} & \tau(\square_0) = n+1\\
    0 & \tau(\square_0) \neq n+1.
     \end{cases} $$
\end{lem}

\subsection{Positive Affine Hecke Algebra}

We will need the following basic notions about affine Hecke algebras in type A.
 
\begin{defn}\label{first presentation of AHA}
    Define the \textit{\textbf{positive affine Hecke algebra}} $\sA_n$ to be the $\mathbb{Q}(q,t)$-algebra generated by $T_1,\ldots, T_{n-1}$ and $\theta_1,\ldots, \theta_n$ subject to the relations
    \begin{itemize}
        \item $T_1,\ldots, T_{n-1}$ generate $\sH_n$
        \item $\theta_i\theta_j = \theta_j\theta_i$ for all $1\leq i,j \leq n$
        \item $\theta_{i+1} = tT_i^{-1}\theta_i T_{i}^{-1}$ for $1\leq i \leq n-1$
        \item $T_i\theta_j = \theta_j T_i$ for $j \notin \{i,i+1\}$
    \end{itemize}

    Define the special elements $\pi_n$ and $\varphi_1,\ldots, \varphi_{n-1}$ of $\sA_n$ by 
    \begin{itemize}
        \item $\pi_n:= t^{n-1}\theta_1T_{1}^{-1}\cdots T_{n-1}^{-1}$
        \item $\varphi_i := (tT_{i}^{-1})\theta_i - \theta_i (tT_{i}^{-1}).$
    \end{itemize}
\end{defn}

\begin{remark}
     Note that the $\theta_i$ elements are distinct from the Cherednik elements $\xi_i$ defined in \cite{DL_2011} which after aligning the different finite Hecke algebra $T_i$ conventions satisfy $\pi_n = \xi_1T_1\ldots T_{n-1}$.
\end{remark}

\begin{defn}\label{map from affine to finite hecke}
    Define the $\mathbb{Q}(q,t)$-algebra homomorphism $\rho_n:\sA_n \rightarrow \sH_n$ by 
    \begin{itemize}
        \item $\rho_n(T_i) = T_i$ for $1\leq i \leq n-1$
        \item $\rho_n(\theta_i) = \overline{\theta}_i$ for $1\leq i \leq n.$
    \end{itemize}
    For a $\sH_n$-module $V$ we will denote by $\rho_n^{*}(V)$ the $\sA_n$-module with action defined for $v \in V$ and $w \in \sA_n$ by $w(v) := \rho_n(w)(v).$
\end{defn}

\begin{remark}
    Note for $\lambda \in \Par$ with $|\lambda| = n$ that $\rho_n^{*}(\lambda)$ is an irreducible $\sA_n$-module with a basis of $\theta$-weight vectors with distinct weights given by $\tau \in \SYT(\lambda).$
\end{remark}

\subsection{Positive Double Affine Hecke Algebra}
Here we describe the positive double affine Hecke algebras in type $GL_n$.

\begin{defn}\label{daha def}
    Define the \textit{\textbf{positive double affine Hecke algebra}} $\sD_n$ to be the $\mathbb{Q}(q,t)$-algebra generated by $T_1,\ldots,T_{n-1}$, $\theta_1,\ldots, \theta_n$, and $X_1,\ldots, X_n$ subject to the relations
    \begin{itemize}
        \item $T_1,\ldots,T_{n-1}$ and $\theta_1,\ldots, \theta_n$ generate $\sA_n$
        \item $X_iX_j = X_jX_i$ for $1\leq i,j \leq n$
        \item $X_{i+1} = tT_i^{-1}X_iT_i^{-1}$ for $1\leq i \leq n-1$
        \item $T_iX_j = X_jT_i$ for $1\leq i\leq n-1$ and $1\leq j \leq n$ with $j \notin \{i,i+1\}$
        \item $\pi_nX_i = X_{i+1}\pi_n$ for $1\leq i \leq n-1$
        \item $\pi_nX_n = qX_1\pi_n.$
    \end{itemize}
\end{defn}

\begin{remark}
    Note that $\sD_n$ has a $\mathbb{Z}_{\geq 0}$-grading determined by $\deg(X_i) = 1$ and $\deg(Y_i) = \deg(T_i) = 0.$
\end{remark}

\begin{defn}\label{spherical DAHA defn}
    Let $\epsilon^{(n)} \in \sH_n$ denote the (normalized) trivial idempotent given by 
    $$\epsilon^{(n)}:= \frac{1}{[n]_{t}!}\sum_{\sigma \in \mathfrak{S}_n} t^{{n\choose 2} - \ell(\sigma)} T_{\sigma}$$ where
    $[n]_{t}!:= \prod_{i=1}^{n}(\frac{1-t^i}{1-t}).$ The \textit{\textbf{positive spherical double affine Hecke algebra}} $\sD_n^{\text{sph}}$ is the (non-unital) subalgebra of $\sD_n$ given by $\sD_n^{\text{sph}}:= \epsilon^{(n)}\sD_n \epsilon^{(n)}.$ 
\end{defn}

\begin{remark}
    Given any $\sD_n$-module V the space $\epsilon^{(n)}(V)$ is naturally a $\sD_n^{\text{sph}}$-module. Note that $\sD_n^{\text{sph}}$ is unital with unit $\epsilon^{(n)}$ and that $\sD_n^{\text{sph}}$ has a grading inherited from $\sD_n.$
\end{remark}

\subsection{Positive Elliptic Hall Algebra}
Here we give a very brief description of the positive elliptic Hall algebra. 

\begin{defn}
    For $\ell > 0$ define the special elements $P_{0,\ell}^{(n)}, P_{\ell,0}^{(n)} \in \sD_n^{\text{sph}}$ by 
    \begin{itemize}
        \item $P_{0,\ell}^{(n)} = \epsilon^{(n)}\left( \sum_{i=1}^{n} \theta_i^{\ell} \right) \epsilon^{(n)}$
        \item $P_{\ell,0}^{(n)} = q^{\ell} \epsilon^{(n)} \left( \sum_{i=1}^{n} X_i^{\ell} \right) \epsilon^{(n)}.$
    \end{itemize}
\end{defn}

\begin{remark}
    Following \cite{SV}, we may also define elements $P_{a,b}^{(n)} \in \sD_n^{\text{sph}}$ for $(a,b) \in \mathbb{Z}^{2}\setminus \{(0,0)\}$ similarly using an algebra automorphism action by $\SL_{2}(\mathbb{Z})$ on $\sD_n^{\text{sph}}$. However, we will not need to work with these general elements directly as $P_{0,\ell}^{(n)}, P_{\ell,0}^{(n)}$ for $\ell >0$ generate $\sD_n^{\text{sph}}.$ 
\end{remark}

\begin{thm}\cite{SV}
    There is a unique graded algebra surjection $\sD_{n+1}^{\text{sph}}\rightarrow \sD_n^{\text{sph}}$ determined for $\ell > 0$ by 
    $ P_{0,\ell}^{(n+1)} \rightarrow P_{0,\ell}^{(n)}$ and $ P_{\ell,0}^{(n+1)} \rightarrow P_{\ell,0}^{(n)}.$
\end{thm}

\begin{defn}\cite{SV}
    The \textit{\textbf{positive elliptic Hall algebra}} $\sE^{+}$ is the inverse limit of the graded algebras $\sD_n^{\text{sph}}$ with respect to the maps $\sD_{n+1}^{\text{sph}}\rightarrow \sD_n^{\text{sph}}.$ For $\ell >0$ define the special elements $P_{0,\ell}:= \lim_n P_{0,\ell}^{(n)}$ and $P_{\ell,0}:= \lim_n P_{\ell,0}^{(n)}.$ 
\end{defn}

The positive elliptic Hall algebra $\sE^{+}$ is generated by $P_{0,\ell}, P_{\ell,0}$ for $\ell >0$ and has $\mathbb{Z}_{\geq 0}$-grading determined by $\deg(P_{0,\ell}) = 0$ and $\deg(P_{\ell,0}) = \ell.$ Remarkably, there is a description of $\sE^{+}$ (and its Drinfeld double $\sE$ called the elliptic Hall algebra) given by straightforward generators and relations \cite{SV} which we will not detail here.

\section{DAHA Modules from Young Diagrams}\label{DAHA Modules from Young Diagrams}
\subsection{The $\sD_n$-module $V_{\lambda}$}
We begin by defining a collection of DAHA modules indexed by Young diagrams $\lambda \in \Par.$ These modules are the same as those appearing in \cite{DL_2011} but we take the approach of using induction from $\sA_n$ to $\sD_n$ for their definition.

 \begin{defn}
     Let $\lambda \in \Par$ with $|\lambda| = n$. Define the $\sD_n$-module $V_{\lambda}$ to be the induced module 
     $$V_{\lambda}:= \Ind_{\sA_n}^{\sD_n}\rho_n^{*}(\lambda).$$
 \end{defn}

These modules naturally have the basis given by $X^{\alpha}\otimes \tau$ where $X^{\alpha}$ is a monomial and $\tau \in \SYT(\lambda).$ Note that the action of $\pi_n$ on $V_{\lambda}$ is invertible so we may consider the action of $\pi_n^{-1}$ although we have not formally included $\pi_n^{-1}$ into the algebra $\sD_n.$

Using the theory of intertwiners for DAHA and some combinatorics we are able to show the following structural results. The $F_{\tau}$ appearing below are the version of the non-symmetric vv. Macdonald polynomials following our conventions.

\begin{prop}\label{weight basis prop}
    There exists a basis of $V_{\lambda}$ consisting of $\theta^{(n)}$-weight vectors $\{F_{\tau}: \tau \in \PSYT(\lambda)\}$ with distinct $\theta^{(n)}$-weights such that the following hold:
    \begin{itemize}
        \item $\theta_i^{(n)}(F_{\tau}) = q^{w_{\tau}(i)}t^{c_{\tau}(i)}F_{\tau}$
        \item If $\tau \in \SYT(\lambda)$ then $F_{\tau} = 1\otimes \tau.$
        \item If $s_i(\tau) > \tau$ then $$\left( tT_{i}^{-1} + \frac{(t-1)q^{w_{\tau}(i+1)}t^{c_{\tau}(i+1)}}{q^{w_{\tau}(i)}t^{c_{\tau}(i)}-q^{w_{\tau}(i+1)}t^{c_{\tau}(i+1)}}\right)F_{\tau} = F_{s_i(\tau)}.$$
        \item $F_{\Psi(\tau)} = q^{w_1(\tau)}X_n\pi_n^{-1}F_{\tau}.$
    \end{itemize}
\end{prop}

\begin{prop}\label{Mackey decomposition}
    The $\sD_n$-module $V_{\lambda}$ has the following decomposition into $\sA_n$-submodules:
    $$\Res^{\sD_n}_{\sA_n} V_{\lambda} = \bigoplus_{T \in \RYT(\lambda)} U_T$$ 
    where $U_T:= \text{span}_{\mathbb{Q}(q,t)}\{F_{\tau}: \mathfrak{p}_{\lambda}(\tau) = T\}.$
    Further, each $\sA_n$-module $U_T$ is irreducible.
\end{prop}

 Using induction on the partial order defined over $\PSYT(\lambda)$ we obtain the following result. Recall Section \ref{defs and nots} for notation.

\begin{prop}
    For $T \in \RYT(\lambda)$, $F_{\Top(T)}$ has a triangular expansion of the form 
    $$F_{\Top(T)}= t^{-b_T}X^{\nu(T)}\otimes S(T) + \sum_{\beta \prec \nu(T)} X^{\beta} \otimes v_{\beta} $$ for some $v_{\beta} \in \lambda.$ Here $\prec$ denotes the Bruhat order on $\mathbb{Z}_{\geq 0}^n.$
\end{prop}

\subsection{Connecting Maps Between $V_{\lambda^{(n)}}$}

In order to build the inverse systems which we will use to define Muraghan-type modules for $\sE^{+}$, we need to consider the following maps.

\begin{defn}\label{connecting maps def}
    Let $\lambda \in \Par$. For $n \geq n_{\lambda}$ define $\Phi^{(n)}_{\lambda}: V_{\lambda^{(n+1)}} \rightarrow V_{\lambda^{(n)}}$ as the $\mathbb{Q}(q,t)$-linear map determined by 
    $$\Phi^{(n)}_{\lambda}(X^{\alpha}\otimes v) = \mathbbm{1}(\alpha_{n+1} = 0) X_1^{\alpha_1}\cdots X_n^{\alpha_n}\otimes \mathfrak{q}^{(n)}_{\lambda}(v).$$
\end{defn}

The next proposition is the most crucial step in proving the main theorem of this paper. Its proof relies heavily on the use of the re-orientated Cherednik operators $\theta_i$ and their spectral analysis.

\begin{prop}\label{stability for locally maximal psyt}
    Let $T \in \RYT(\lambda^{(n)})$ and $T' \in \RYT(\lambda^{(n+1)})$ be such that $T(\square) = T'(\square)$ for $ \square \in \lambda^{(n)}$ and $T'(\square_0) = 0$ for $\square_0 \in \lambda^{(n+1)}/\lambda^{(n)}.$
    Then $$\Phi^{(n)}_{\lambda}(F_{\Top(T')})= F_{\Top(T)}.$$
\end{prop}

The maps $\Phi^{(n)}_{\lambda}$ possess another remarkable property regarding the action of the Macdonald elements which will be required later in this paper.

\begin{cor}\label{connecting maps and symmetrization of thetas}
    For all $\ell > 0$ and $n \geq n_{\lambda}$, 
    $$\Phi^{(n)}_{\lambda}\left( P_{(0,\ell)}^{(n+1)} - \sum_{\square \in \lambda^{(n+1)}} t^{\ell c(\square)}  \right) = \left( P_{(0,\ell)}^{(n)} - \sum_{\square \in \lambda^{(n)}} t^{\ell c(\square)} \right) \Phi^{(n)}_{\lambda}.$$
\end{cor}

\section{Positive EHA Representations from Young Diagrams}\label{Positive EHA Representations from Young Diagrams}

In this section we build $\sE^{+}$-modules using the maps $\Phi^{(n)}_{\lambda}$ and the stability of the $F_{\tau}$ basis already described.

\subsection{The $\sD_n^{\text{sph}}$-modules $W_{\lambda^{(n)}}$}
Here we consider the spherical subspaces of the $V_{\lambda}$ modules.

 \begin{defn}\label{sph DAHA reps defn}
     For $\lambda \in \Par$ with $|\lambda| = n$ define the $\sD_n^{\text{sph}}$-module $W_{\lambda}:= \epsilon^{(n)}(V_{\lambda}).$
 \end{defn}

 We will need the following combinatorial description of the AHA submodules of $V_{\lambda}$ which contain a nonzero $T_i$-invariant vector.

 \begin{prop}\label{sym AHA submodules}
 For $\lambda \in \Par$ with $|\lambda| = n$ and $T \in \RYT(\lambda)$, 
 $$\mathrm{dim}_{\mathbb{Q}(q,t)} \epsilon^{(n)}(U_T) =\begin{cases}
    1 & T \in \RSSYT(\lambda)\\
    0 & T \notin \RSSYT(\lambda).
     \end{cases}  $$
 \end{prop}

 We define the vv. symmetric Macdonald polynomials in the following way. These will align up to a scalar with those in \cite{DL_2011}.

\begin{defn}\label{Macdonald polynomial def}
    Let $T \in \RSSYT(\lambda).$ Define $P_{T} \in \epsilon^{(n)}(U_T)$ to be the unique element of the form 
    $$P_T = F_{\Top(T)} + \sum_{y\in \PSYT(\lambda;T)\setminus \{\Top(T)\}} \kappa_{y}F_y.$$
\end{defn}

We can now use Prop. \ref{weight basis prop} and Prop. \ref{stability for locally maximal psyt} to prove the following results for the $P_T.$ 

\begin{prop}\label{expansion of sym into nonsym}
    For all $T\in \RSSYT(\lambda)$, 
    $$P_T = \sum_{\tau \in \PSYT(\lambda;T)} \prod_{(\square_1,\square_2) \in \Inv(\tau)} \left( \frac{q^{T(\square_1)}t^{c(\square_1) +1} - q^{T(\square_2)}t^{c(\square_2) }}{q^{T(\square_1)}t^{c(\square_1)} - q^{T(\square_2)}t^{c(\square_2)}} \right) F_{\tau}.$$
\end{prop}

\begin{prop}\label{simple spectrum}
    The set $\{P_T: T \in \RSSYT(\lambda)\}$ is a $\mathbb{Q}(q,t)[\theta_1,\ldots, \theta_n]^{\mathfrak{S}_n}$-weight basis for $W_{\lambda}$. Further, 
    $$P_{0,\ell}^{(n)}(P_T) = \left( \sum_{\square \in \lambda} q^{\ell T(\square)}t^{\ell c(\square)}  \right) P_T.$$
\end{prop}

\begin{cor}\label{stability for MacD poly}
Let $T \in \RSSYT(\lambda^{(n)})$ and $T' \in \RSSYT(\lambda^{(n+1)})$ such that $T(\square) = T'(\square)$ for $ \square \in \lambda^{(n)}$ and $T'(\square_0) = 0$ for $\square_0 \in \lambda^{(n+1)}/\lambda^{(n)}.$
Then
$\Phi^{(n)}_{\lambda}(P_{T'})= P_{T}.$
\end{cor}

\subsection{Stable Limit of the $W_{\lambda^{(n)}}$} 

We now can define the stable-limit spaces $\widetilde{W}_{\lambda}$ and the generalized symmetric Macdonald functions.

\begin{defn}\label{stable limit defn}
    Let $\lambda \in \Par$.  Define the infinite diagram $\lambda^{(\infty)}:= \bigcup_{n \geq n_{\lambda}} \lambda^{(n)}.$ Define $\Omega(\lambda)$ to be the set of all labellings $T:\lambda^{(\infty)} \rightarrow \mathbb{Z}_{\geq 0}$ such that 
    \begin{itemize}
        \item $|\{ \square \in \lambda^{(\infty)}: T(\square) \neq 0\}| < \infty$
        \item $T$ increases weakly along rows
        \item $T$ increases strictly along columns.
    \end{itemize}

    Define the space $W_{\lambda}^{(\infty)}$ to be the inverse limit $\varprojlim W_{\lambda^{(n)}}$ with respect to the maps $\Phi_{\lambda}^{(n)}.$ Let $\widetilde{W}_{\lambda}$ be the subspace of all bounded $X$-degree elements of $W_{\lambda}^{(\infty)}$. For any symmetric function $F \in \Lambda$ define $F[X]^{\bullet}$ to be the corresponding multiplication operator on $\widetilde{W}_{\lambda}$. Lastly, for $T \in \Omega(\lambda)$ define the generalized symmetric Macdonald function $\MacD_T:= \lim_{n} P_{T|_{\lambda^{(n)}}} \in \widetilde{W}_{\lambda}.$
\end{defn}

\begin{remark}
    Each $\MacD_T$ is homogeneous of degree 
$ \mathrm{deg}(\MacD_T) = \sum_{\square \in \lambda^{(\infty)}} T(\square) < \infty.$ It is clear from definition that the set of all $\MacD_T$ for $T \in \Omega(\lambda)$ gives a $\mathbb{Q}(q,t)$-basis of $\widetilde{W}_{\lambda}.$ Lastly, the multiplication operators $F[X]^{\bullet}$ are well-defined since $\Phi_{\lambda}^{(n)}X_{n+1} = 0.$
\end{remark}

\begin{defn}\label{MacD operator defn}
    For $\ell > 0 $ define the operator $\Delta_{\ell}:\widetilde{W}_{\lambda} \rightarrow \widetilde{W}_{\lambda}$ to be the stable-limit 
    
    $\Delta_{\ell}:= \lim_n \left(P_{0,\ell}^{(n)} - \sum_{\square \in \lambda^{(n)}} t^{\ell c(\square)} \right).$
\end{defn}
    
\subsection{$\sE^{+}$ Action on $\widetilde{W}_{\lambda}$}

Finally, we are ready to state and prove the main result of this paper.

\begin{thm}[Main Theorem] \label{main theorem}
    For $\lambda \in \Par$, $\widetilde{W}_{\lambda}$ is a graded $\sE^{+}$-module with action determined for $\ell>0$ by
    \begin{itemize}
        \item $P_{\ell,0} \rightarrow q^{\ell} p_{\ell}[X]^{\bullet}$
        \item $P_{0,\ell} \rightarrow \Delta_{\ell}$.
    \end{itemize}
     Further, $\widetilde{W}_{\lambda}$ is spanned by a basis of eigenvectors $\{ \MacD_{T}\}_{T \in \Omega(\lambda)}$ with distinct eigenvalues for the Macdonald operator $\Delta = \Delta_1$.
\end{thm}
\begin{proof}
It suffices to establish that the map $\sE^{+} \rightarrow \End_{\mathbb{Q}(q,t)}(\widetilde{W}_{\lambda})$ satisfies the generating relations of $\sE^{+}.$ Any such relation is a non-commutative polynomial expression in $\sE^{+}$ of the form $$F(P_{0,1},\ldots,P_{0,r},P_{1,0},\ldots,P_{s,0}) = 0$$ for some $r,s >0.$ By an argument of Schiffmann-Vasserot (Lemma 1.3 in \cite{SV}), there are automorphisms $\Gamma^{(n)}$ of $\sD_n^{\text{sph}}$ such that $\Gamma^{(n)}(P^{(n)}_{0,\ell})= P^{(n)}_{0,\ell} - \sum_{\square \in \lambda^{(n)}} t^{\ell c(\square)}$ and $\Gamma^{(n)}(P^{(n)}_{\ell,0}) = P^{(n)}_{\ell,0}.$ By applying the canonical quotient maps $\Pi_n: \widetilde{W}_{\lambda} \rightarrow W_{\lambda^{(n)}}$ we see using Cor. \ref{connecting maps and symmetrization of thetas} that as maps 

\begin{align*}
\Pi_nF(P_{0,1},\ldots,P_{0,r},P_{1,0},\ldots,P_{s,0})& =F(\Gamma^{(n)}(P^{(n)}_{0,1}),\ldots,\Gamma^{(n)}(P^{(n)}_{0,r}),\Gamma^{(n)}(P^{(n)}_{1,0}),\ldots,\Gamma^{(n)}(P^{(n)}_{s,0}))\Pi_n \\
&= \Gamma^{(n)}(F(P^{(n)}_{0,1},\ldots,P^{(n)}_{0,r},P^{(n)}_{1,0},\ldots,P^{(n)}_{s,0})) \Pi_n = 0.\\
\end{align*}

As this holds for all $n \geq n_{\lambda}$, it follows that $F(P_{0,1},\ldots,P_{0,r},P_{1,0},\ldots,P_{s,0}) = 0$ in $\End_{\mathbb{Q}(q,t)}(\widetilde{W}_{\lambda})$ as desired. The last statement regarding the spectrum of $\Delta$ follows directly from Prop. \ref{simple spectrum} and Cor. \ref{stability for MacD poly}.
\end{proof}

\begin{remark}
    For $\lambda = \emptyset$, $\widetilde{W}_{\emptyset} = \Lambda$ recovers the standard representation of $\sE^{+}.$ In this case, $\Omega(\emptyset) = \Par$ and $\MacD_{\mu} = P_{\mu}[X;q^{-1},t]$ (up to nonzero scalar).
\end{remark}

By considering the grading of each module $\widetilde{W}_{\lambda}$ and the spectral theory of the Macdonald operator $\Delta$ we can prove the following. 

\begin{prop}
    For $\lambda, \mu \in \Par$ distinct, $\widetilde{W}_{\lambda} \ncong \widetilde{W}_{\mu}$ as graded $\mathcal{E}^{+}$-modules. 
\end{prop}

\begin{remark}
    Although we will not detail the construction here, there is a natural way to extend the $\sE^{+}$ action on each $\widetilde{W}_{\lambda}$ to an action of the full elliptic Hall algebra $\sE$ using a non-degenerate $q,t$-sesquilinear form.
\end{remark}

\section{Pieri Rule}\label{Pieri Rule section}
In this section we give the description of a Pieri rule for the generalized symmetric Macdonald functions $\MacD_{T}$. We need to consider the following $q,t$-rational function.

\begin{defn}
    For $T \in \RSSYT(\lambda)$ define 
     $$K_T(q,t):= \frac{[\mu(T)]_{t}!}{[n]_{t}!} \prod_{(\square_1,\square_2) \in \I(T)} \left( \frac{q^{T(\square_1)}t^{c(\square_1)} - q^{T(\square_2)}t^{c(\square_2) +1 }}{q^{T(\square_1)}t^{c(\square_1)} - q^{T(\square_2)}t^{c(\square_2)}} \right).$$
\end{defn}

Using Prop. \ref{expansion of sym into nonsym} and some book-keeping we obtain the following finite-rank Pieri formula.

\begin{thm}\label{Pieri Rule for Finite vars}
    For $T \in \RSSYT(\lambda)$ and $1 \leq r \leq n$ we have the expansion 
    $$e_r[X_1+\ldots + X_n]P_T = \sum_{S} d^{(r)}_{S,T} P_{S}$$ where 

\begin{equation*}
    \begin{split}
        \frac{d^{(r)}_{S,T}}{t^{ {r\choose 2}}e_r(1,\ldots,t^{n-1})K_{S}(q,t)} =
     \sum_{\substack{\tau \in \PSYT(\lambda;T) \\ \text{s.t.}\\ \Psi^r(\tau) \in \PSYT(\lambda;S)}} 
         t^{c_{\tau}(1)+\ldots + c_{\tau}(r)}
        \prod_{(\square_1,\square_2) \in \Inv(\tau)} \left( \frac{q^{T(\square_1)}t^{c(\square_1) +1} - q^{T(\square_2)}t^{c(\square_2)}}{q^{T(\square_1)}t^{c(\square_1)} - q^{T(\square_2)}t^{c(\square_2)}} \right) ~\times \\ 
        \prod_{(\square_1,\square_2) \in \Inv(\Psi^r(\tau))} \left( \frac{q^{S(\square_1)}t^{c(\square_1)} - q^{S(\square_2)}t^{c(\square_2)}}{q^{S(\square_1)}t^{c(\square_1)} - q^{S(\square_2)}t^{c(\square_2) +1}} \right)
    \end{split}
\end{equation*}
and $S$ ranges over all $S \in \RSSYT(\lambda)$ one can obtain from $T$ by adding $r$ $1$'s to the boxes of $T$ with at most one $1$ being added to each box.
\end{thm}

\begin{defn}\label{Pieri coeff defn}
    For $S,T \in \Omega(\lambda)$ and $r \geq 1$ define $\mathfrak{d}^{(r)}_{S,T} \in \mathbb{Q}(q,t)$ by 
    $$e_r[X]^{\bullet}(\MacD_T) = \sum_{S \in \Omega(\lambda)} \mathfrak{d}^{(r)}_{S,T} \MacD_{S}.$$ Define the \textit{\textbf{rank}} $\rk(T)$ to be the minimal $n \geq n_{\lambda}$ such that $T|_{\lambda^{(\infty)}\setminus\lambda^{(n)}} = 0.$ 
\end{defn}

\begin{remark}
    Note that from Theorem \ref{Pieri Rule for Finite vars} it is clear for $T \in \Omega(\lambda)$ and $ r\geq 1$ that each $S \in \Omega(\lambda)$ such that $\mathfrak{d}^{(r)}_{S,T} \neq 0$ will necessarily be obtained from $T$ by adding $r$ $1$'s to the boxes of $T$ with at most one $1$ being added to each box. As such the set of $S$ with $\mathfrak{d}^{(r)}_{S,T} \neq 0$ is finite.
\end{remark}
We can use the stability from Cor. \ref{stability for MacD poly} to obtain a Pieri rule.

\begin{cor}[Pieri Rule]\label{Pieri Rule}
     Let $S,T \in \Omega(\lambda)$ and $r \geq 1$. For all $n \geq \rk(T)+r$
     $$\mathfrak{d}^{(r)}_{S,T} = d^{(r)}_{S|_{\lambda^{(n)}},T|_{\lambda^{(n)}}}.$$
\end{cor}

\section{Family of Product-Series Identities}\label{Product-Sum Identities Sections}
In order to state the final result of this paper we need the following.
\begin{defn}\label{asymptotic peridoic standard tableaux defn}
    A non-negative \textit{\textbf{asymptotic periodic standard Young tableau}} with base shape $\lambda \in \Par$ is a labelling $\tau: \lambda^{(\infty)} \rightarrow \{ iq^{a}: i \geq 1, a \geq 0\}$ such that
    \begin{itemize}
        \item $\tau$ is strictly increasing along rows and columns
        \item The set of boxes $\square \in \lambda^{(\infty)}$ such that $\tau(\square) = iq^a$ for some $i \geq 1$ and $a > 0$ is finite.
        \item For all $ i \geq 1$ there exists a unique $\square \in \lambda^{(\infty)}$ such that $\tau(\square) = iq^a$ for some $a \geq 0.$
    \end{itemize}
    We will write $\APSYT(\lambda)$ for the set of all non-negative asymptotic periodic standard Young tableaux with base shape $\lambda \in \Par.$ If $\tau \in \APSYT(\lambda)$ has that for every $\square \in \lambda$, $\tau(\square) = iq^0$ for some $i \geq 1$ then we will call $\tau$ an \textit{\textbf{asymptotic standard Young tableau}} with base shape $\lambda.$ We will write $\ASYT(\lambda)$ for the set of asymptotic standard Young tableau with base shape $\lambda.$ As an abuse of notation will write $\mathfrak{p}_{\lambda}: \APSYT(\lambda) \rightarrow \Omega(\lambda)$ for the map given on $\tau \in \APSYT(\lambda)$ by $\mathfrak{p}_{\lambda}(\tau)(\square) = a$ whenever $\tau(\square) = iq^a$ for some $i \geq 1 $. We will let $\APSYT(\lambda;T)$ denote the set of all $\tau \in \APSYT(\lambda)$ with $\mathfrak{p}_{\lambda}(\tau)=T.$ 
\end{defn}

\begin{example}
 \ytableausetup{centertableaux, boxframe= normal, boxsize= 1.85em}\begin{ytableau}
4q^3&5q^3&6q^3&2q^2&3q^2&12q^1&13q^0&14q^0&15q^0& \none [\dots]\\
 1q^2& 6q^2 & 11q^1    \\
 8q^1& 9q^1 & \none \\
 10q^0& \none & \none \\
 \end{ytableau} $\in \APSYT(3,2,1)$
\end{example}

\begin{defn}\label{more APSYT defn}
    For $T \in \Omega(\lambda)$ define $S(T) \in \ASYT(\lambda)$ by ordering the boxes of $\lambda^{(\infty)}$ according to $\square_1 \leq \square_2$ if and only if 
\begin{itemize}
    \item $T(\square_1) > T(\square_2)$ or
    \item $T(\square_1) = T(\square_2)$ and $\square_1$ comes before $\square_2$ in the column-standard labelling of $\lambda^{(\infty)}.$
\end{itemize}
     Let $\tau \in \APSYT(\lambda;T).$
     An ordered pair of boxes $(\square_1,\square_2) \in \lambda^{(\infty)} \times \lambda^{(\infty)}$ is called an \textit{\textbf{inversion pair}} of $\tau$ if $S(T)(\square_1) < S(T)(\square_2)$ and $i > j$ where $\tau(\square_1) = iq^a$, $\tau(\square_2) = jq^b$ for some $a,b \geq 0.$
    The set of all inversion pairs of $\tau$ will be denoted by $\Inv(\tau)$ and we will write $\inv(\tau) = |\Inv(\tau)|.$ Define the \textit{\textbf{rank}} $\rk(\tau)$ to be the minimal $n \geq n_{\lambda}$ such that $\tau|_{\lambda^{(\infty)}\setminus \lambda^{(n)}}$ has consecutive labels. We will write $\mu_T:= \mu(T|_{\lambda^{(\rk(T))}})$ (see Def. \ref{decomposing RYT into SYT and partition}).
\end{defn}

Using a $t$-adic convergence argument and a limiting version of Cor. \ref{stability for MacD poly} we can show the following.
\begin{thm}\label{prod-sum formula}
For $T \in \Omega(\lambda)$ we have the following equality in $\mathbb{Q}(q)((t)):$
\begin{align*}
    \frac{\prod_{\square \in \lambda^{(\rk(T))}}\left( 1 -q^{-T(\square)}t^{\rk(T)-|\lambda|-c(\square)} \right)}{(1-t)^{\rk(T)}[\mu_T]_{t}!} &\prod_{(\square_1,\square_2) \in \I(\lambda^{(\rk(T))})} \left( \frac{1 - q^{T(\square_2) -T(\square_1)}t^{c(\square_2)-c(\square_1) }}{1 - q^{T(\square_2) -T(\square_1)}t^{c(\square_2) -c(\square_1)+1}}         \right)  \\
    &= \sum_{\tau \in \APSYT(\lambda;T)} t^{\inv(\tau)} \prod_{(\square_1,\square_2) \in \Inv(\tau)} \left( \frac{1 - q^{T(\square_2)-T(\square_1)}t^{c(\square_2)-c(\square_1) -1}}{1 - q^{T(\square_2)-T(\square_1)}t^{c(\square_2)-c(\square_1) +1}} \right).\\
\end{align*}
\end{thm}

\begin{example}
If $\lambda = \emptyset$ and 
$T =$ 
\ytableausetup{centertableaux, boxframe= normal, boxsize= 2.25em}\begin{ytableau}
1&0&0&\none [\dots]\\
\end{ytableau} $\in \Omega(\emptyset)$
then from Thm. \ref{prod-sum formula} we get 
$$\frac{1-q^{-1}t}{1-t} = \sum_{k=0}^{\infty} t^k \prod_{j=1}^{k} \left( \frac{1-q^{-1}t^{j-1}}{1-q^{-1}t^{j+1}}\right).$$
\end{example}

\printbibliography

\end{document}